\documentclass[a4paper,11pt]{amsart}

\usepackage{latexsym}
\usepackage{amssymb}
\usepackage{amsthm}
\usepackage{amscd}
\usepackage{amsmath}
\usepackage[dvipdfmx]{graphicx}
\usepackage[all]{xy}
\usepackage{multicol}
\newtheorem{Theorem}{Theorem}[section]

\newtheorem{Corollary}[Theorem]{Corollary}
\newtheorem{Proposition}[Theorem]{Proposition}

\newtheorem{Question}[Theorem]{Question}

\def\fka{{\frak a}}

\def\fkm{{\frak m}}

\def\opn#1#2{\def#1{\operatorname{#2}}}

\opn\Spec{Spec}
\opn\Supp{Supp}
\opn\supp{supp}
\opn\Max{Max}
\opn\max{max}
\opn\Min{Min}
\opn\min{min}
\opn\Ass{Ass}
\opn\Assh{Assh}
\opn\Ann{Ann}
\opn\depth{depth}
\opn\rank{rank}

\opn\Mat{Mat}

\opn\Tot{Tot}

\opn\Sym{Sym}

\def\bsn{{\boldsymbol n}}

\opn\div{div}
\opn\Div{Div}
\opn\cl{cl}
\opn\Cl{Cl}

\opn\Ker{Ker}
\opn\Coker{Coker}
\opn\Im{Im}
\opn\Hom{Hom}
\opn\Tor{Tor}
\opn\Ext{Ext}
\opn\End{End}
\opn\Fitt{Fitt}
\opn\Aut{Aut}
\opn\id{id}

\opn\nat{nat}
\opn\pff{pf}
\opn\Pf{Pf}
\opn\GL{GL}
\opn\SL{SL}
\opn\G{G}
\opn\E{E}
\opn\H{H}
\opn\M{M}
\opn\mod{mod}
\opn\ord{ord}
\opn\det{det}
\opn\Soc{Soc}

\opn\chara{char}
\opn\length{\ell}
\opn\pd{pd}
\opn\rk{rk}
\opn\projdim{proj\,dim}
\opn\injdim{inj\,dim}
\opn\rank{rank}
\opn\depth{depth}
\opn\grade{grade}
\opn\height{ht}
\opn\embdim{emb\,dim}
\opn\codim{codim}

\renewcommand{\hat}{\widehat}

\title{A formula for the associated Buchsbaum-Rim multiplicities of a direct sum of cyclic modules II} 
\author{Futoshi Hayasaka}
\address{Department of Environmental and Mathematical Sciences, Okayama University, 
3-1-1 Tsushimanaka, Kita-ku, Okayama, 700-8530, JAPAN}
\email{hayasaka@okayama-u.ac.jp}
\keywords{Buchsbaum-Rim multiplicity, Buchsbaum-Rim function, cyclic modules, Hilbert-Samuel multiplicity}
\subjclass[2010]{Primary 13H15; Secondary 13P99}

\begin{document}

\maketitle

\begin{abstract}
The associated Buchsbaum-Rim multiplicities of a module are a descending sequence 
of non-negative integers. These invariants of a module are a generalization of the classical 
Hilbert-Samuel multiplicity of an ideal.   
In this article, we compute the associated Buchsbaum-Rim multiplicity of 
a direct sum of cyclic modules and give a formula for the second to last positive
Buchsbaum-Rim multiplicity in terms of the ordinary Buchsbaum-Rim and 
Hilbert-Samuel multiplicities. This is a natural generalization of a formula given 
by Kirby and Rees.  
\end{abstract}

%%%%%%%%%%%%%%%%%%%%%%%%%%%%%%%%%%%%%%%%%%%%%%%%%%%%%
\section{Introduction}
%%%%%%%%%%%%%%%%%%%%%%%%%%%%%%%%%%%%%%%%%%%%%%%%%%%%%

Let $(R, \fkm)$ be a Noetherian local ring with the maximal ideal $\fkm$ of dimension 
$d>0$. 
The associated Buchsbaum-Rim multiplicities of an $R$-module $C$ of finite length, 
which is denoted by $\{ e^j(C) \}_{0 \leq j \leq d+r-1}$, are 
a sequence of integers. These are invariants of $C$ introduced by 
Kleiman-Thorup \cite{KT2} and Kirby-Rees \cite{KR2} independently. 
For an $R$-module $C$ of finite length with a minimal free presentation 
$R^n \stackrel{\varphi}{\to} R^r \to C \to 0$, 
the multiplicities are defined by the so-called Buchsbaum-Rim function of two variables
$$\Lambda(p, q):={\ell}_R(S_{p+q}/M^{p}S_{q}), $$
where $S_p$ (resp. $M^p$) is a homogeneous component of 
degree $p$ of $S=\Sym_R(F)$ (resp. $R[M]=\Im \Sym_R(\varphi)$). 
The function $\Lambda(p, q)$ is eventually a polynomial of total degree $d+r-1$, and then
the associated Buchsbaum-Rim multiplicities are defined as 
for $j=0, 1, \dots , d+r-1$, 
$$e^j(C):=(\mbox{The coefficient of} \ p^{d+r-1-j}q^j \ \mbox{in the polynomial})\times (d+r-1-j)!j!. $$
These are a descending sequence of non-negative integers with $e^{r-1}(C)$ is positive, and  
$e^j(C)=0$ for $j \geq r$. This was proved by Kleiman-Thorup \cite{KT2} and Kirby-Rees \cite{KR2} independently. Moreover, they proved that the first multiplicity $e^0(C)$ coincides with 
the ordinary Buchsbaum-Rim multiplicity $e(C)$ of $C$ introduced in \cite{BR2}, 
which is the normalized leading coefficient of the polynomial function 
$\lambda(p)=\Lambda(p, 0)=\ell_R(S_p/M^p)$ of degree $d+r-1$ for $p \gg 0$. Namely,  
$$e(C) = e^0(C) \geq e^1(C) \geq \dots \geq e^{r-1}(C)>e^r(C)= \dots = e^{d+r-1}(C)=0. $$
Note that the ordinary Buchsbaum-Rim multiplicity $e(R/I)$ of a cyclic module defined by an
$\fkm$-primary ideal $I$ 
coincides with the classical Hilbert-Samuel multiplicity $e(I)$ of $I$. 
Thus, the ordinary Buchsbaum-Rim multiplicity $e^0(C)=e(C)$ and the associated one 
$e^j(C)$ are a generalization of the classical Hilbert-Samuel multiplicity. However, as 
compared to the classical Hilbert-Samuel multiplicity, the Buchsbaum-Rim multiplicities 
are not well-understood. 

There are some cases where the computation of the ordinary Buchsbaum-Rim multiplicity 
is possible (see \cite{Bi, CLU, J, KR1, KR2} for instance). In particular, in the case where  
$C$ is a direct sum of cyclic modules, there is an interesting relation 
between the ordinary Buchsbaum-Rim multiplicity and the mixed multiplicities of ideals. 
Let $I_1, \dots , I_r$ be $\fkm$-primary ideals in $R$. Then 
Kirby and Rees proved that

$$e(R/I_1 \oplus \dots \oplus R/I_r)=\sum_{\stackrel{i_1, \dots , i_r \geq 0}{i_1+\dots +i_r=d}}e_{i_1 \cdots i_r}(I_1, \dots , I_r), $$
where $e_{i_1 \cdots i_r}(I_1, \dots , I_r)$ is the mixed multiplicity of $I_1, \dots , I_r$ of type $(i_1, \dots , i_r)$ (see \cite{KR1, KR2} and also \cite{Bi}). Then we are interested in the other associated Buchsbaum-Rim multiplicities in this case. 

The starting point of this research is the following interesting formula which was also discovered by 
Kirby-Rees \cite{KR1, KR2}. Suppose that $I_1 \subset I_2 \subset \dots \subset I_r$. 
Then they proved that for any $j=1, \dots , r-1$, 
the $j$th Buchsbaum-Rim multiplicity can be 
expressed as the ordinary Buchsbaum-Rim multiplicity of a direct sum of $(r-j)$ 
cyclic modules defined by the last $(r-j)$ ideals: 
$$e^j(R/I_1 \oplus \dots \oplus R/I_r)=e(R/I_{j+1} \oplus \dots \oplus R/I_r).  $$
In particular, the last positive one $e^{r-1}$ can be expressed as the classical Hilbert-Samuel multiplicity $e(I_r)$ of the largest ideal: 
$$e^{r-1}(R/I_1 \oplus \dots \oplus R/I_r)=e(R/I_r). $$

Then it is natural to ask the formula for $e^j(R/I_1 \oplus \dots \oplus R/I_r)$
without the assumption $I_1 \subset I_2 \subset \dots \subset I_r$. 
However, as compared to the special case considered in \cite{KR1, KR2}, 
it seems that the problem is more complicated, and we need a different approach
to obtain the formula in general. 
Recently, we tried to compute the function $\Lambda(p, q)$ directly using some ideas and 
obtained the formula for $e^{r-1}(R/I_1 \oplus \dots \oplus R/I_r)$ without the assumption
$I_1 \subset \dots \subset I_r$. 
Indeed, we proved in our previous work \cite[Theorem 1.3]{Ha2} that 
for any $\fkm$-primary ideals $I_1, \dots , I_r$, the last positive Buchsbaum-Rim multiplicity 
can be expressed as the classical Hilbert-Samuel multiplicity $e(I_1+\dots +I_r)$ 
of the sum of all ideals: 
$$e^{r-1}(R/I_1 \oplus \dots \oplus R/I_r)=e(R/[I_1 + \dots + I_r]).  $$

The present purpose is to improve 
the method of computation given in \cite{Ha2} towards a formula for not only the last positive 
Buchsbaum-Rim multiplicity $e^{r-1}(R/I_1 \oplus \dots \oplus R/I_r)$ but also 
the next one $e^{r-2}(R/I_1 \oplus \dots \oplus R/I_r)$ in terms of the ordinary Buchsbaum-Rim and Hilbert-Samuel multiplicities. 
Here is the main result. 

\begin{Theorem}\label{main}
Let $I_1, \dots , I_r$ be arbitrary $\fkm$-primary ideals in $R$. Then we have a formula
$$e^{r-2}(R/I_1 \oplus \dots \oplus R/I_r)=E_{r-1}(I_1, \dots , I_r)-(d+1)(r-1)e(R/[I_1 + \dots + I_r]), $$
where $E_{r-1}(I_1, \dots , I_r)$ is a sum of the ordinary Buchsbaum-Rim multiplicities of two cyclic modules defined by the ideals $I_1+\dots +\hat{I_j}+\dots +I_r$ and $I_1+\dots+ I_r: $ 
$$E_{r-1}(I_1, \dots , I_r):=\sum_{j=1}^r e(R/[I_1+\dots +\hat{I_j}+\dots +I_r] 
\oplus R/[I_1+\dots+ I_r]). $$
\end{Theorem}

Let me illustrate the formula when $r=3$. Let $C=R/I_1\oplus R/I_2 \oplus R/I_3$. 
It is known that $e^0(C)$ coincides with the ordinary Buchsbaum-Rim multiplicity by \cite{KR2, KT2}, and $e^2(C)$ can be expressed as the ordinary Hilbert-Samuel multiplicity of the sum of all ideals by \cite{Ha2}. Theorem \ref{main} tells us that there is a similar expression for 
the remaining multiplicity $e^1(C)$. Namely, if we put 
$I_{123}:=I_1+I_2+I_3$ and $I_{ij}:=I_i+I_j$ for $1 \leq i < j \leq 3$, then we can write 
all the multiplicities in terms of ordinary Buchsbaum-Rim multiplicities and hence 
mixed multiplicities.

\begin{align*}
e^0(C) &= e(R/I_1\oplus R/I_2 \oplus R/I_3)  & \\
e^1(C) &=e(R/I_{23} \oplus R/I_{123})
+e(R/I_{13} \oplus R/I_{123})
+e(R/I_{12} \oplus R/I_{123}) 
-2(d+1) e(R/I_{123}) & \\ 
e^2(C) &= e(R/I_{123}) .  & \\
\end{align*}

Our formula can be viewed as a natural generalization 
of the above mentioned 
Kirby-Rees formula for $e^{r-2}(R/I_1 \oplus \dots \oplus R/I_r)$
in a special case where $I_1 \subset I_2 \subset \dots \subset I_r$. 
Indeed, as an immediate consequence of Theorem \ref{main}, we get the following. 

\setcounter{section}{4}
\setcounter{Theorem}{1}
\begin{Corollary}\label{cor}
Let $I_1, \dots , I_r$ be $\fkm$-primary ideals in $R$ and assume that 
$I_1, \dots , I_{r-1} \subset I_r$, that is, the ideal $I_r$ is the largest 
ideal. Then we have a formula
$$e^{r-2}(R/I_1 \oplus \dots \oplus R/I_r)=e(R/[I_1+\dots +I_{r-1}] \oplus R/I_r). $$
In particular, if $I_1 \subset I_2 \subset \dots \subset I_r$, then 
$$e^{r-2}(R/I_1 \oplus \dots  \oplus R/I_r)=e(R/I_{r-1} \oplus R/I_r). $$
\end{Corollary}
\setcounter{section}{1}

The contents of the article are organized as follows. 
In the next section 2, we will recall some necessary notation and results 
from our previous work \cite{Ha2}. In section 3, 
we will compute the Buchsbaum-Rim function of two variables by improving the method in \cite{Ha2}. 
In the last section 4, we will give a proof of Theorem \ref{main} and its consequence 
Corollary \ref{cor}.  
We will also discuss the remaining multiplicities $e^j(C)$ 
for $j=1, \dots , r-3$. 

Throughout this article, we will work in the same manner in our previous work \cite{Ha2}. 
Let $(R, \fkm)$ be a Noetherian local ring with the maximal ideal 
$\fkm$ of dimension $d>0$. Let $r>0$ be a fixed positive integer and let 
$[r]=\{1, \dots , r\}$. For a finite set $A$, 
${}^{\sharp} A$ denotes the number of elements of $A$. 
Vectors are always written in bold-faced letters, e.g., $\boldsymbol i =(i_1, \dots , i_r)$. 
We work in the usual multi-index notation. Let $I_1, \dots , I_r$ be ideals in $R$ and 
let $t_1, \dots , t_r$ be indeterminates. 
Then for a vector $\boldsymbol i =(i_1, \dots , i_r) \in \mathbb Z_{\geq 0}^r$, 
we denote 
$\boldsymbol I^{\boldsymbol i}=I_1^{i_1} \cdots I_r^{i_r}, \boldsymbol t^{\boldsymbol i}=t_1^{i_1} 
\cdots t_r^{i_r}$ and $| \boldsymbol i | =i_1+ \dots + i_r$. 
For vectors $\boldsymbol a, \boldsymbol b \in \mathbb Z^r$, 
$\boldsymbol a \geq \boldsymbol b \stackrel{{\rm def}}{\Leftrightarrow} 
a_i \geq b_i \ \mbox{for all} \ i=1, \dots , r.$
Let $\boldsymbol 0=(0, \dots , 0)$ be the zero vector in $\mathbb Z_{\geq 0}^r$. 
By convention, empty sum is defined to be zero. 

%%%%%%%%%%%%%%%%%%%%%%%%%%%%%%%%%%%%%%%%%%%%%%%%%%%%%
\section{Preliminaries}
%%%%%%%%%%%%%%%%%%%%%%%%%%%%%%%%%%%%%%%%%%%%%%%%%%%%%
In this section, we give a few elementary facts to compute the associated Buchsbaum-Rim multiplicities. 
See also \cite[section 2]{Ha2} for the related facts and the details. 

In what follows, let $I_1, \dots , I_r$ be $\fkm$-primary ideals in $R$ and 
let $C=R/I_1 \oplus \dots \oplus R/I_r$. Let $S=R[t_1, \dots , t_r]$ be a polynomial ring over $R$ 
and let $R[M]=R[I_1t_1, \dots , I_rt_r]$ be 
the multi-Rees algebra of $I_1, \dots , I_r$. Let 
$S_p$ (resp. $M^p$) be a homogeneous component of 
degree $p$ of $S$ (resp. $R[M]$).
Then it is easy to see that the function $\Lambda(p, q)$ can be expressed as 
$${\displaystyle \Lambda(p, q) = \sum_{\boldsymbol n \in H_{p,q}} 
\ell_R(R/J_{p, q}({\boldsymbol n})) }$$
where $H_{p, q}:=\{ \boldsymbol n \in \mathbb Z_{\geq 0}^r \mid |\boldsymbol n |=p+q \}$ and 
${\displaystyle J_{p,q}({\boldsymbol n}):=\sum_
{\substack{| \boldsymbol i|=p \\ \boldsymbol 0 \leq \boldsymbol i \leq \boldsymbol n}} \boldsymbol I^
{\boldsymbol i}
}$ for $\boldsymbol n \in H_{p, q}$. 
For a subset $\Delta \subset H_{p, q}$, we set 
$$\Lambda_{\Delta}(p, q):=\sum_{\boldsymbol n \in \Delta} 
\ell_R(R/J_{p, q}({\boldsymbol n})). $$
As in \cite{Ha2}, we consider the following special subsets of $H_{p, q}$, 
which will play a basic role in our computation of $\Lambda(p, q)$. 
For $p, q>0$ and $k=1, \dots , r$, let 
$$\Delta_{p, q}^{(k)}:=\{\boldsymbol n \in H_{p, q} \mid n_1, \dots , n_k>p, n_{k+1}+ \dots + n_r \leq p \}. $$
Then the function $\Lambda_{\Delta_{p, q}^{(k)}}(p, q)$ can be described explicitly as follows. 

\

\begin{Proposition}\label{2.1} $($\cite[Proposition 2.3]{Ha2}$)$
Let $p, q>0$ with $q \geq (p+1)r$ and let $k=1, \dots , r$. Then 
$$\Lambda_{\Delta_{p, q}^{(k)}}(p, q)=\sum_{\stackrel{n_{k+1}, \dots , n_r \geq 0}{n_{k+1}+ \dots +n_r \leq p}} 
{q-(k-1)p-1-(n_{k+1}+\dots +n_r) \choose k-1} \ell_R(R/\fka), $$
where $\fka$ is an ideal depending on $n_{k+1}, \dots , n_r: $  
$$\displaystyle{
\fka:=(I_1+\dots +I_k)^{p-(n_{k+1}+\dots +n_r)} \prod_{j=k+1}^r(I_1+\dots +I_k+I_j)^{n_j}}. $$ 
\end{Proposition}

Here we make a slightly different description of the above mentioned basic functions 
$\Lambda_{\Delta_{p, q}^{(k)}}(p, q)$. 
To state it, we first recall some elementary facts about the ordinary
Buchsbaum-Rim functions and multiplicities of a direct sum of cyclic modules. 
The ordinary Buchsbaum-Rim function $\lambda(p)$ of $C=R/I_1 \oplus \dots \oplus R/I_r$
(we will often denote it $\lambda_C(p)$ to emphasize the defining module $C$) can be 
expressed as follows: 
\begin{eqnarray*}
\lambda(p)&=&\ell_R(S_p/M^p)\\
&=&\sum_
{\substack{\boldsymbol i \geq \boldsymbol 0 \\ | \boldsymbol i|=p}} 
\ell_R(R/\boldsymbol I^
{\boldsymbol i}) \\
&=&\sum_
{\substack{\boldsymbol i \geq \boldsymbol 0 \\ | \boldsymbol i|=p}} 
\ell_R(R/I_1^{i_1} \cdots I_r^{i_r}). 
\end{eqnarray*}
In particular, if we consider the case where $I_1=\dots = I_r=:I$, then 
$$\lambda(p)={p+r-1 \choose r-1}\ell_R(R/I^p). $$
The function $\ell_R(R/I^p)$ is just the Hilbert-Samuel function of $I$ so that 
it is a polynomial for all large enough $p$, and one can write 
$$\ell_R(R/I^p)=\frac{e(R/I)}{d!}p^d+(\mbox{lower terms}), $$
where $e(R/I)$ is the usual Hilbert-Samuel multiplicity of $I$. 
Therefore, the ordinary Buchsbaum-Rim function can be expressed as
$$\lambda(p)=\frac{e(R/I)}{d!(r-1)!}p^{d+r-1}+(\mbox{lower terms}). $$
This implies the following elementary formula for 
the ordinary Buchsbaum-Rim multiplicity: 
\begin{equation}\label{ordinaryBuchsbaum-Rim}
e(C)=e(\underbrace{R/I \oplus \dots \oplus R/I}_{r})={d+r-1 \choose r-1}e(R/I). 
\end{equation}

Now, let me give another description of $\Lambda_{\Delta_{p, q}^{(k)}}(p, q)$.

\begin{Proposition}\label{2.2}
Let $p, q>0$ with $q \geq (p+1)r$ and let $k=1, \dots , r$. Then 
\begin{multline*}
\Lambda_{\Delta_{p, q}^{(k)}}(p, q)=
{q-(k-1)p-1 \choose k-1}\lambda_{L_k}(p) \\
-\sum_{\stackrel{n_{k+1}, \dots , n_r \geq 0}{n_{k+1}+ \dots +n_r \leq p}} 
\sum_{i=0}^{n_{k+1}+\dots +n_r-1}{q-(k-1)p-2-i \choose k-2} \ell_R(R/\fka),  
\end{multline*}
where $\displaystyle{L_k:=R/[I_1+\dots +I_k] \oplus \bigoplus_{j=k+1}^{r}R/[I_1+\dots +I_k+I_j]}$ 
is a direct sum of $(r-k+1)$ cyclic modules and 
$\displaystyle{\fka:=(I_1+\dots +I_k)^{p-(n_{k+1}+\dots +n_r)} \prod_{j=k+1}^r(I_1+\dots +I_k+I_j)^{n_j}}$ is an ideal depending on $n_{k+1}, \dots , n_r$. 
\end{Proposition}

\begin{proof}
The case where $k=1$ follows from Proposition \ref{2.1}. Indeed, 
\begin{eqnarray*}
\Lambda_{\Delta_{p, q}^{(1)}}(p, q) &=&
\sum_{\stackrel{n_{2}, \dots , n_r \geq 0}{n_{2}+ \dots +n_r \leq p}} \ell_R\Big(
R \big/I_1^{p-(n_2+\dots +n_r)} \prod_{j=2}^r (I_1+I_j)^{n_j} \Big) \\
&=& \sum_{\substack{\boldsymbol i \geq \boldsymbol 0 \\ | \boldsymbol i|=p}} \ell_R
\big( R/I_1^{i_1} (I_1+I_2)^{i_2} \cdots (I_1+I_r)^{i_r} \big) \\
&=&\lambda_{L_1}(p). 
\end{eqnarray*}
Suppose that $k \geq 2$. 
By using an elementary combinatorial formula ${m-\ell \choose n}={m \choose n}-\sum_{i=0}^{\ell-1}
{m-\ell+i \choose n-1}$, one can see that 
\begin{eqnarray*}
&&{q-(k-1)p-1-(n_{k+1}+\dots +n_r) \choose k-1}\\
&=&{q-(k-1)p-1 \choose k-1}-\sum_{j=0}^{n_{k+1}+\dots +n_r-1}{q-(k-1)p-1-(n_{k+1}+\dots+ n_r)+j 
\choose k-2} \\
&=&{q-(k-1)p-1 \choose k-1}-\sum_{j=0}^{n_{k+1}+\dots +n_r-1}{q-(k-1)p-2+j-(n_{k+1}+\dots +n_r-1) 
\choose k-2} \\
&=&{q-(k-1)p-1 \choose k-1}-\sum_{i=0}^{n_{k+1}+\dots +n_r-1}{q-(k-1)p-2-i 
\choose k-2}.  
\end{eqnarray*}
By Proposition \ref{2.1}, we can write the function $\Lambda_{\Delta_{p, q}^{(k)}}(p, q)$ 
as follows: 
\begin{eqnarray*}
&&\Lambda_{\Delta_{p, q}^{(k)}}(p, q)\\
&=&\sum_{\stackrel{n_{k+1}, \dots , n_r \geq 0}{n_{k+1}+ \dots +n_r \leq p}} 
{q-(k-1)p-1-(n_{k+1}+\dots +n_r) \choose k-1} \ell_R(R/\fka)\\
&=&\sum_{\stackrel{n_{k+1}, \dots , n_r \geq 0}{n_{k+1}+ \dots +n_r \leq p}} 
\left[ {q-(k-1)p-1 \choose k-1}-\sum_{i=0}^{n_{k+1}+\dots +n_r-1}{q-(k-1)p-2-i
\choose k-2} \right] \ell_R(R/\fka)\\
\end{eqnarray*}
\begin{eqnarray*}
&=&
{q-(k-1)p-1 \choose k-1} \sum_{\stackrel{n_{k+1}, \dots , n_r \geq 0}{n_{k+1}+ \dots +n_r \leq p}}
\ell_R(R/\fka) \\
&&\hspace{4cm}
-\sum_{\stackrel{n_{k+1}, \dots , n_r \geq 0}{n_{k+1}+ \dots +n_r \leq p}}
\sum_{i=0}^{n_{k+1}+\dots +n_r-1}{q-(k-1)p-2-i \choose k-2} \ell_R(R/\fka) \\
&=&{q-(k-1)p-1 \choose k-1} \lambda_{L_k}(p) \\
&&\hspace{4cm}
-\sum_{\stackrel{n_{k+1}, \dots , n_r \geq 0}{n_{k+1}+ \dots +n_r \leq p}}
\sum_{i=0}^{n_{k+1}+\dots +n_r-1}{q-(k-1)p-2-i \choose k-2} \ell_R(R/\fka), \\
\end{eqnarray*}
where 
$\displaystyle{L_k:=R/[I_1+\dots +I_k] \oplus \bigoplus_{j=k+1}^{r}R/[I_1+\dots +I_k+I_j]}$ 
is a direct sum of $(r-k+1)$ cyclic modules and 
$\displaystyle{\fka:=(I_1+\dots +I_k)^{p-(n_{k+1}+\dots +n_r)} \prod_{j=k+1}^r(I_1+\dots +I_k+I_j)^{n_j}}$
 is an ideal depending on $n_{k+1}, \dots , n_r$. 
\end{proof}

%%%%%%%%%%%%%%%%%%%%%%%%%%%%%%%%%%%%%%%%%%%%%%%%%%%%%
\section{A computation of the Buchsbaum-Rim functions}
%%%%%%%%%%%%%%%%%%%%%%%%%%%%%%%%%%%%%%%%%%%%%%%%%%%%%
In this section, we compute the function $\Lambda(p,q)$ by improving the method in \cite{Ha2}
towards a formula for $e^{r-2}(R/I_1\oplus \dots \oplus R/I_r)$. 
The notation we will use here is under the same manner in \cite{Ha2}. 
See also \cite[Section 3]{Ha2} for more detailed observations. 

In order to compute the multiplicity defined by the asymptotic function 
$\Lambda(p, q)$, 
we may assume that 
$q \geq (p+1)r \gg 0.$ 
In what follows, let $p, q$ be fixed integers satisfying $q \geq (p+1)r \gg 0$. 
We put $H:=H_{p, q}$ for short. 
Then the set $H$ can be divided by $r$-regions 
$$H=\coprod_{k=1}^r H^{(k)}, $$
where $H^{(k)}:=\{ \boldsymbol n \in H \mid {}^{\sharp} \{ i \mid n_i > p \}=k \}. $
Moreover, we divide each $H^{(k)}$ into ${r \choose k}$-regions 
$$H^{(k)}=\coprod_{\stackrel{A \subset [r]}{{}^{\sharp}A=r-k}} D_A^{(k)}, $$
where $D_A^{(k)}:=\{ \bsn \in H^{(k)} \mid n_i >p \ \mbox{for} \ i \notin A, n_i \leq p \ \mbox{for} \ i \in A \}$ and $D_{\emptyset}^{(r)}=H^{(r)}$. 
Then 
$$H=\coprod_{k=1}^r \coprod_{\stackrel{A \subset [r]}{{}^{\sharp}A=r-k}} D_A^{(k)}. $$

Let me illustrate this decomposition when $r=3$. 
Figure \ref{pic1} below 
is the picture where $H^{(3)}=D_{\emptyset}^{(3)}$ is the region of the pattern of dots,  
$H^{(2)}=D_{\{1\}}^{(2)} \coprod D_{\{2\}}^{(2)} \coprod D_{\{3\}}^{(2)}$ is the region of no pattern, 
and $H^{(1)}=D_{\{1, 2\}}^{(1)} \coprod D_{\{1, 3\}}^{(1)} \coprod D_{\{2, 3\}}^{(1)}$ is the region
of lines.

\begin{figure}[h]

\includegraphics[clip, trim=30 500 0 80]{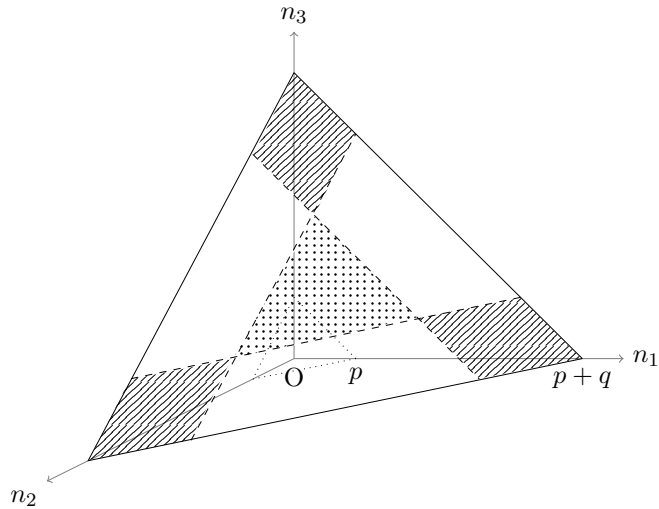}
%\includegraphics[clip, trim=20 510 20 100]{Fig1.pdf}

%\begin{tikzpicture}[x=0.3cm, y=0.3cm]
%\coordinate[label=below:{O}]  (O) at (0,0);
%\coordinate[label=below:{$p$}]  (P1) at (3,0);
%\coordinate (P2) at (0,3);
%\coordinate (P3) at (-2,-1);
%\coordinate[label=below:{$p+q$}]  (Q1) at (14,0);
%\coordinate (Q2) at (0,14);
%\coordinate (Q3) at (-10,-5);
%\coordinate[label=right:{$n_1$}] (N1) at (16,0);
%\coordinate[label=above:{$n_3$}] (N2) at (0,16);
%\coordinate[label=below left:{$n_2$}]  (N3) at (-12,-6);
%\coordinate (A1) at (11,3);
%\coordinate (A2) at (9,-1);
%\coordinate (B1) at (-2,10);
%\coordinate (B2) at (3,11);
%\coordinate (C1) at (-5,-4);
%\coordinate (C2) at (-8,-1);
%\coordinate (R1) at (6,2);
%\coordinate (R2) at (1,7);
%\coordinate (R3) at (-3,0);
%\draw[gray, ->] (O) -- (N1);
%\draw[gray, ->] (O) -- (N2);
%\draw[gray, ->] (O) -- (N3);
%\draw[dotted] (P1) -- (P2) -- (P3) -- (P1);
%\filldraw [draw=white, pattern = dots] (R1) -- (R2) -- (R3) -- (R1);
%\filldraw [draw=white, pattern = north east lines] (A1) -- (R1) -- (A2) -- (Q1) -- (A1);
%\filldraw [draw=white, pattern =north east lines] (B1) -- (R2) -- (B2) -- (Q2) -- (B1);
%\filldraw [draw=white, pattern = north east lines] (C1) -- (R3) -- (C2) -- (Q3) -- (C1);
%\draw[dashed] (A1) -- (C2);
%\draw[dashed] (B1) -- (A2);
%\draw[dashed] (C1) -- (B2);
%\draw (Q1) -- (Q2) -- (Q3) -- (Q1);
%\end{tikzpicture}

\caption{A decomposition of $H$ when $r=3$}\label{pic1}

\end{figure}

Therefore, the computation of $\Lambda(p, q)$ can be reduced to the one of 
each $\Lambda_{D_A^{(k)}} (p, q)$:
\begin{eqnarray*}
\Lambda(p, q)&=&\sum_{k=1}^r \Lambda_{H^{(k)}}(p, q)\\
&=&\sum_{k=1}^r  \sum_{\stackrel{A \subset [r]}{{}^{\sharp}A=r-k}} \Lambda_{D_A^{(k)}} (p, q). 
\end{eqnarray*}
When $k=r$, $D_{\emptyset}^{(r)}=H^{(r)}=\Delta_{p, q}^{(r)}$ so that 
we get the explicit description of $\Lambda_{H^{(r)}}(p, q)$ by Proposition \ref{2.2}. 
Similarly, when $k=r-1$, $D_{\{r\}}^{(r-1)}=\Delta_{p, q}^{(r-1)}$ so that 
we get the explicit description of $\Lambda_{D_{\{r\}}^{(r-1)}}(p, q)$
 by Proposition \ref{2.2} and hence the one of $\Lambda_{H^{(r-1)}}(p, q)$.

\begin{Proposition}\label{3.1} We have the following description of $\Lambda_{H^{(k)}}(p, q)$
when $k=r, r-1$.  
\begin{enumerate}
\item The case where $k=r: $  
$$\Lambda_{H^{(r)}} (p, q)={q-(r-1)p-1 \choose r-1} \lambda_{L}(p), $$
where $L:=R/[I_1+\dots +I_r]$ is a cyclic module. 
\item The case where $k=r-1: $  
\begin{multline*}
\Lambda_{H^{(r-1)}} (p, q)={q-(r-2)p-1 \choose r-2} \sum_{j=1}^r \lambda_{L_j}(p)  \\
-\sum_{j=1}^r  \sum_{n=0}^p \sum_{i=0}^{n-1} {q-(r-2)p-2-i \choose r-3} \ell_R(R/\fka_j(n))
\end{multline*}
where
$L_j:=R/[I_1+\dots +\hat{I_j}+ \dots +I_r] \oplus R/[I_1+\dots +I_r]$ is a direct sum of two cyclic modules
and 
$\displaystyle{\fka_j(n):=(I_1+\dots +\hat{I_j}+ \dots +I_r)^{p-n} (I_1+\dots +I_r)^n}$
is an ideal depending on $j$ and $n$.
\end{enumerate} 
\end{Proposition}

\begin{proof}
These follow directly from Proposition \ref{2.2}. 
\end{proof}

We now turn to investigate the remaining functions $\Lambda_{H^{(k)}}(p, q)$ when $k=1, 2, \dots, r-2$. 
These cases seem to be more complicated than the case of $k=r, r-1$. 
Suppose that $k=1, 2, \dots , r-2$ and let $A$ be a subset of $[r]$ with ${}^{\sharp} A=r-k$. 
Then we divide the set $D_A^{(k)}$ into $2$-parts as follows: 
$$D_A^{(k)}=E_{A-}^{(k)} \coprod E_{A+}^{(k)}, $$
where 
$$E_{A-}^{(k)}:=\{ \bsn \in D_A^{(k)} \mid \sum_{i \in A} n_i \leq p \}, $$
$$E_{A+}^{(k)}:=\{ \bsn \in D_A^{(k)} \mid \sum_{i \in A} n_i  > p \}. $$
Let 
$$H_{-}^{(k)}:=\coprod_{\stackrel{A \subset [r]}{{}^{\sharp}A=r-k}} E_{A-}^{(k)},$$
$$H_{+}^{(k)}:=\coprod_{\stackrel{A \subset [r]}{{}^{\sharp}A=r-k}} E_{A+}^{(k)}. $$ 
Then 
$$\Lambda_{H^{(k)}}(p, q)=\Lambda_{H_{-}^{(k)}}(p, q)+\Lambda_{H_{+}^{(k)}}(p, q). $$

Let me illustrate this decomposition when $r=3$. Figure \ref{pic2} below 
is the picture where 
$H_{-}^{(1)}=E_{\{1, 2\}-}^{(1)} \coprod E_{\{1, 3\}-}^{(1)} \coprod E_{\{2, 3\}-}^{(1)}$
is the region of the pattern of lines, and 
$H_{+}^{(1)}=E_{\{1, 2\}+}^{(1)} \coprod E_{\{1, 3\}+}^{(1)} \coprod E_{\{2, 3\}+}^{(1)}$
is the region of the pattern of dots. 

\

\begin{figure}[h]

\includegraphics[clip, trim=30 500 0 80]{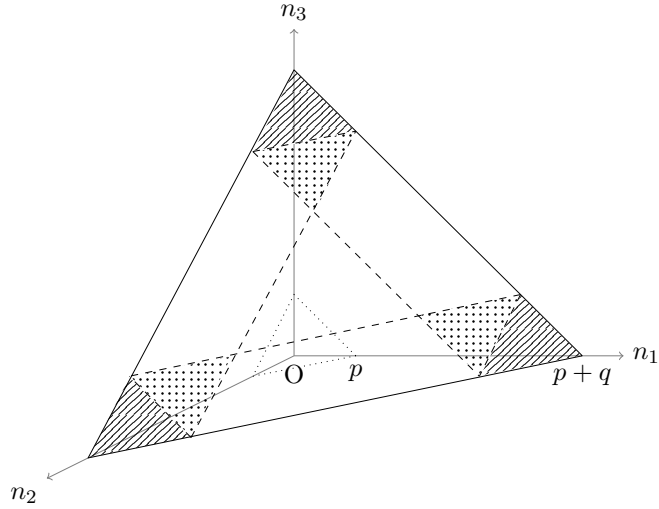}
%\includegraphics[clip, trim=20 510 20 100]{Fig2.pdf}

%\begin{tikzpicture}[x=0.3cm, y=0.3cm]
%\coordinate[label=below:{O}]  (O) at (0,0);
%\coordinate[label=below:{$p$}]  (P1) at (3,0);
%\coordinate (P2) at (0,3);
%\coordinate (P3) at (-2,-1);
%\coordinate[label=below:{$p+q$}]  (Q1) at (14,0);
%\coordinate (Q2) at (0,14);
%\coordinate (Q3) at (-10,-5);
%\coordinate[label=right:{$n_1$}] (N1) at (16,0);
%\coordinate[label=above:{$n_3$}] (N2) at (0,16);
%\coordinate[label=below left:{$n_2$}]  (N3) at (-12,-6);
%\coordinate (A1) at (11,3);
%\coordinate (A2) at (9,-1);
%\coordinate (B1) at (-2,10);
%\coordinate (B2) at (3,11);
%\coordinate (C1) at (-5,-4);
%\coordinate (C2) at (-8,-1);
%\coordinate (R1) at (6,2);
%\coordinate (R2) at (1,7);
%\coordinate (R3) at (-3,0);
%\draw[gray, ->] (O) -- (N1);
%\draw[gray, ->] (O) -- (N2);
%\draw[gray, ->] (O) -- (N3);
%\draw[dotted] (P1) -- (P2) -- (P3) -- (P1);
%\filldraw [draw=white, pattern = north east lines] (A1) -- (A2) -- (Q1) -- (A1);
%\filldraw [draw=white, pattern =north east lines] (B1) --  (B2) -- (Q2) -- (B1);
%\filldraw [draw=white, pattern = north east lines] (C1) -- (C2) -- (Q3) -- (C1);
%\filldraw [draw=white, pattern = dots] (A1) -- (R1) -- (A2) -- (A1);
%\filldraw [draw=white, pattern =dots] (B1) --  (R2) -- (B2) -- (B1);
%\filldraw [draw=white, pattern = dots] (C1) -- (R3) -- (C2) -- (C1);
%\draw[dashed] (A1) -- (A2);
%\draw[dashed] (B1) -- (B2);
%\draw[dashed] (C1) -- (C2);
%\draw[dashed] (A1) -- (C2);
%\draw[dashed] (B1) -- (A2);
%\draw[dashed] (C1) -- (B2);
%\draw (Q1) -- (Q2) -- (Q3) -- (Q1);
%\end{tikzpicture}

\caption{A decomposition of $H^{(1)}$ when $r=3$}\label{pic2}

\end{figure}

Here we note that $E_{\{k+1, \dots , r\}-}^{(k)}=\Delta_{p, q}^{(k)}$ for 
any $k=1, 2, \dots , r-2$. Thus, 
the function $\Lambda_{H_{-}^{(k)}}(p, q)$ can be expressed explicitly as follows, similar to 
the one of $\Lambda_{H^{(r)}}(p, q)$ and $\Lambda_{H^{(r-1)}}(p, q)$. 

\begin{Proposition}\label{3.2}
For any $k=1, 2, \dots, r-2$, we have the following description. 
\begin{multline*}
\Lambda_{H_{-}^{(k)}} (p, q)
={q-(k-1)p-1 \choose k-1} 
 \sum_{\stackrel{A \subset [r]}{{}^{\sharp}A=r-k}} 
\lambda_{L_A}(p)  \\
-\sum_{\stackrel{A \subset [r]}{{}^{\sharp}A=r-k}}  
\sum_{\stackrel{n_j \geq 0 (j \in A)}{(\sum_{j \in A} n_j) \leq p}} 
\sum_{i=0}^{(\sum_{j \in A} n_j)-1} {q-(k-1)p-2-i \choose k-2} \ell_R(R/\fka), 
\end{multline*}
where 
$
\displaystyle{
L_A:=\bigg( 
R\Big/
\Big[\sum_{s \in [r] \setminus A} I_s \Big]
\bigg)
\oplus 
\bigoplus_{j \in A}
\bigg(
R \Big/ \Big[ \sum_{s \in [r] \setminus A} I_s +I_j \Big]
\bigg)
}
$ 
is a direct sum of $(r-k+1)$ cyclic modules and 
$\displaystyle{\fka:=\Big(\sum_{s \in [r] \setminus A} I_s \Big)^{p-(\sum_{j \in A} n_j)} 
\prod_{j \in A}\Big(\sum_{s \in [r] \setminus A} I_s+I_j \Big)^{n_j}}$ is an ideal 
depending on $A$ and $n_j$ $(j \in A)$.  
\end{Proposition}

\begin{proof}
This follows directly from Proposition \ref{2.2}. 
\end{proof}

On the other hand, the function $\Lambda_{H_{+}^{(k)}}(p, q)$ seems to be more complicated than 
the one $\Lambda_{H_{-}^{(k)}}(p, q)$. 
We do not get the explicit description, but we have the following inequality.

\begin{Proposition}\label{3.3}
For any $k=1, 2, \dots , r-2$, there exists a polynomial $g^{\circ}_{k}(X) \in \mathbb Q[X]$ 
of degree $d+r-k$ 
such that $$\Lambda_{H_{+}^{(k)}}(p, q) \leq {q-(k-1)p-1 \choose k-1}g^{\circ}_{k}(p). $$
\end{Proposition}

\begin{proof}
This follows from \cite[Lemma 3.5]{Ha2}. 
\end{proof}

Here we consider the following functions $g_k(p)$ and $h_k(p, q)$ appeared in 
Propositions \ref{3.1} and \ref{3.2}, 
which will be used in the next section. For any $k=1, \dots , r-1$, we define 

\begin{align}
g_k(p)&:=\sum_{\stackrel{A \subset [r]}{{}^{\sharp}A=r-k}} \lambda_{L_A}(p) \label{polyg} \\
h_k(p, q)&:=\sum_{\stackrel{A \subset [r]}{{}^{\sharp}A=r-k}} 
\sum_{\stackrel{n_j \geq 0 (j \in A)}{(\sum_{j \in A} n_j) \leq p}} 
\sum_{i=0}^{(\sum_{j \in A} n_j)-1} {q-(k-1)p-2-i \choose k-2} \ell_R(R/\fka)
\end{align}
where 
$
\displaystyle{
L_A:=\bigg( 
R\Big/
\Big[\sum_{s \in [r] \setminus A} I_s \Big]
\bigg)
\oplus 
\bigoplus_{j \in A}
\bigg(
R \Big/ \Big[ \sum_{s \in [r] \setminus A} I_s+I_j \Big]
\bigg)
}
$ 
is a direct sum of $(r-k+1)$ cyclic modules and 
$\displaystyle{\fka:=\Big(\sum_{s \in [r] \setminus A} I_s \Big)^{p-(\sum_{j \in A} n_j)} 
\prod_{j \in A}\Big(\sum_{s \in [r] \setminus A} I_s+I_j \Big)^{n_j}}$. 
When $k=r$, we set $g_r(p)=\lambda_{R/[I_1+\dots +I_r]}(p)$ and $h_r(p, q)=0$. 
Note that for $p, q \gg 0$, $g_k(p)$ is a polynomial function of degree
$d+r-k$, and $h_k(p, q)$ is a non-negative integer valued function.

Then, the above two Propositions \ref{3.2} and \ref{3.3} imply the following. 

\begin{Corollary}\label{3.4}
For any $k=1, 2, \dots , r-2$, there exists a polynomial $f_{k}(X) \in \mathbb Q[X]$ 
of degree $d+r-k$ 
such that $$\Lambda_{H^{(k)}}(p, q) \leq {q-(k-1)p-1 \choose k-1}f_{k}(p). $$
\end{Corollary}

\begin{proof}
By Propositions \ref{3.2} and \ref{3.3}, 
\begin{eqnarray*}
\Lambda_{H^{(k)}}(p, q)&=&\Lambda_{H_{-}^{(k)}}(p, q)+\Lambda_{H_{+}^{(k)}}(p, q) \\
&\leq& {q-(k-1)p-1 \choose k-1} g_k(p)-h_k(p, q)+{q-(k-1)p-1 \choose k-1} g_k^{\circ}(p) \\
&=&{q-(k-1)p-1 \choose k-1} \big(g_k(p)+g^{\circ}_k(p)\big)-h_k(p, q)\\
&\leq &{q-(k-1)p-1 \choose k-1} (g_k(p)+g^{\circ}_k(p)). 
\end{eqnarray*}
Thus, $f_k(X):=g_k(X)+g_k^{\circ}(X)$ is our desired polynomial. 
\end{proof}

%%%%%%%%%%%%%%%%%%%%%%%%%%%%%%%%%%%%%%%%%%%%%%%%%%%%%
\section{Proof of Theorem \ref{main}}
%%%%%%%%%%%%%%%%%%%%%%%%%%%%%%%%%%%%%%%%%%%%%%%%%%%%%
We give a proof of Theorem \ref{main}. 
In this section, we work in the same situation and under the same notation as in the previous sections. 
For $k=1, 2, \dots , r$, we consider the following function: 
$$F_k(p, q):=\Lambda(p, q)-\sum_{i=1}^k {q-(r-i)p-1 \choose r-i} g_{r-i+1}(p), $$
which is a polynomial function for $p, q \gg 0$ with the total degree is at most 
$d+r-1$. We begin with the following. 

\begin{Proposition}\label{limit}
Suppose that $p$ is a large enough fixed integer. Then 
$$\lim_{q \to \infty} \frac{1}{q^{r-2}} F_2(p, q)=0. $$
\end{Proposition}

\begin{proof}
Fix $p \gg 0$. By Proposition \ref{3.1} and Corollary \ref{3.4}, 
we have the following equalities and inequality.  
\begin{eqnarray*}
F_2(p, q)+h_{r-1}(p, q)&=&\Lambda(p, q)-\Lambda_{H^{(r)}}(p, q)-\Lambda_{H^{(r-1)}}(p, q)\\
&=&\sum_{k=1}^{r-2}\Lambda_{H^{(k)}}(p, q)\\
&\leq & \sum_{k=1}^{r-2} {q-(k-1)p-1 \choose k-1} f_k(p). 
\end{eqnarray*}
Hence, we have that 
$$-h_{r-1}(p, q) \leq F_2(p, q) \leq \sum_{k=1}^{r-2} {q-(k-1)p-1 \choose k-1} f_k(p).  $$
Therefore, it is enough to show that 
\begin{align}
& \lim_{q \to \infty} \frac{1}{q^{r-2}} \sum_{k=1}^{r-2} {q-(k-1)p-1 \choose k-1} f_k(p)=0, \ 
\mbox{and} \label{lim1}\\
& \lim_{q \to \infty} \frac{1}{q^{r-2}} h_{r-1}(p, q)=0.\label{lim2}
\end{align}
The first assertion (\ref{lim1}) is clear because the degree of 
a polynomial function 
$$\sum_{k=1}^{r-2} {q-(k-1)p-1 \choose k-1} f_k(p)$$ 
with respect to $q$ is at most $(r-2)-1=r-3$. 
We show the second assertion (\ref{lim2}). Then 
one can see that 
\begin{eqnarray*}
h_{r-1}(p, q) &=& 
\sum_{j=1}^r  \sum_{n=0}^p \sum_{i=0}^{n-1} {q-(r-2)p-2-i \choose r-3} \ell_R(R/\fka_j(n)) \\
&\leq & 
\sum_{j=1}^r  \sum_{n=0}^p n {q-(r-2)p-2 \choose r-3} \ell_R(R/\fka_j(n)) \\
&\leq & 
\sum_{j=1}^r  \sum_{n=0}^p p {q-(r-2)p-2 \choose r-3} \ell_R(R/\fka_j(n)) \\
&= & 
p {q-(r-2)p-2 \choose r-3} \sum_{j=1}^r  \sum_{n=0}^p \ell_R(R/\fka_j(n)),  \\
\end{eqnarray*}
where
$\displaystyle{\fka_j(n):=(I_1+\dots +\hat{I_j}+ \dots +I_r)^{p-n} (I_1+\dots+ I_r)^n}$.
Note that 
$$\sum_{j=1}^r \sum_{n=0}^p \ell_R(R/\fka_j(n))=\sum_{j=1}^r \lambda_{L_j}(p)$$
is a sum of the ordinary Buchsbaum-Rim functions of two cyclic modules, 
where
$$L_j=R/[I_1+\dots +\hat{I_j}+\dots +I_r]
\oplus R/[I_1+\dots +I_r].  
$$ 
Hence, noting that $h_{r-1}(p, q) \geq 0$, we have that
$$0 \leq h_{r-1}(p, q) \leq {q-(r-2)p-2 \choose r-3} u(p)$$
for some polynomial function $u(p)$ of degree $(d+1)+1=d+2$. 
Therefore,  
$$\lim_{q \to \infty} \frac{1}{q^{r-2}} {q-(r-2)p-2 \choose r-3} u(p)=0$$ 
so that $\lim_{q \to \infty} \frac{1}{q^{r-2}} h_{r-1}(p, q)=0. $
\end{proof}

We are now ready to prove Theorem \ref{main}. 

\begin{proof}[Proof of Theorem \ref{main}]
The degree of $\Lambda(p, q)$ with respect to $q$ is at most $r-1$ so that 
one can write
$$\Lambda(p, q)=\sum_{i=0}^{r-1} a_i q^i $$
where each $a_i$ is a polynomial function of $p$ with degree at most $d+r-1-i$. 
Similarly, we can write
\begin{align}
{q-(r-1)p-1 \choose r-1} g_r(p)&=\sum_{j=0}^{r-1} b_j q^j \notag \\
{q-(r-2)p-1 \choose r-2} g_{r-1}(p)&=\sum_{k=0}^{r-2} c_k q^k \notag
\end{align}
where each $b_j$ (resp. $c_k$) is a polynomial function of $p$ with degree at most $d+r-1-j$ (resp. $d+r-1-k$). Then 
$$F_2(p, q)=(a_{r-1}-b_{r-1})q^{r-1}+(a_{r-2}-b_{r-2}-c_{r-2})q^{r-2}+ (\mbox{lower terms in} \ q). $$
By Proposition \ref{limit}, we have the equalities as polynomials of $p$, 
\begin{align}
a_{r-1}&=b_{r-1}, \ \mbox{and} \label{eq1} \\
a_{r-2}&=b_{r-2}+c_{r-2}. \label{eq2}
\end{align}
Note that the first equality (\ref{eq1}) implies a formula 
$e^{r-1}(C)=e(R/[I_1+\dots +I_r])$ which is 
our previous result in \cite{Ha2}. 
We then look at the second equality (\ref{eq2}). 
Since the total degree $\Lambda(p, q)$ is $d+r-1$, and the coefficient of $p^{d+1}q^{r-2}$ is non-zero, 
which is $\frac{e^{r-2}(C)}{(d+1)!(r-2)!}$, the polynomial 
$a_{r-2}$ is of the form: 
$$a_{r-2}=\frac{e^{r-2}(C)}{(d+1)!(r-2)!}p^{d+1}+(\mbox{lower terms in } p). $$
Since $g_r(p)=\lambda_{R/[I_1+\dots +I_r]}(p)$ is the Hilbert-Samuel function of $I_1+\dots +I_r$, 
\begin{eqnarray*}
&&{q-(r-1)p-1 \choose r-1} g_r(p) \\
&=&{q-(r-1)p-1 \choose r-1} \left( \frac{e(R/[I_1+\dots +I_r])}{d!}p^d+(\mbox{lower terms in } p )\right) \\
&=&\frac{(q-(r-1)p)^{r-1}}{(r-1)!} \cdot \frac{e(R/[I_1+\dots +I_r])}{d!}p^d+(\mbox{lower terms}) \\
&=&\frac{e(R/[I_1+\dots +I_r])}{d!(r-1)!} p^dq^{r-1} 
-\frac{(r-1)e(R/[I_1+\dots +I_r])}{d!(r-2)!}p^{d+1}q^{r-2}+(\mbox{lower terms in $q$})
\end{eqnarray*}
so that 
$$b_{r-2}=-\frac{(r-1)e(R/[I_1+\dots +I_r])}{d!(r-2)!}p^{d+1}. $$
Similarly, since $g_{r-1}(p)=\sum_{j=1}^r \lambda_{L_j}(p)$, and 
its normalized leading coefficient is 
$$E_{r-1}:=E_{r-1}(I_1, \dots , I_r):=\sum_{j=1}^r e(L_j), $$
where  
$$L_j= R/[I_1+\dots +\hat{I_j}+\dots +I_r] \oplus R/[I_1+\dots +I_r], $$
we have that 
\begin{eqnarray*}
{q-(r-2)p-1 \choose r-2} g_{r-1}(p) 
&=&{q-(r-2)p-1 \choose r-2} \left( \frac{E_{r-1}}{(d+1)!}p^{d+1}+(\mbox{lower terms in $p$}) \right) \\
&=&\frac{(q-(r-2)p)^{r-2}}{(r-2)!} \cdot \frac{E_{r-1}}{(d+1)!}p^{d+1}+(\mbox{lower terms}) \\
&=&\frac{E_{r-1}}{(d+1)!(r-2)!} p^{d+1}q^{r-2}+(\mbox{lower terms in $q$}). 
\end{eqnarray*}
Therefore, we get that 
$$c_{r-2}=\frac{E_{r-1}}{(d+1)!(r-2)!} p^{d+1}. $$
By comparing the coefficient of $p^{d+1}$ in the equation (\ref{eq2}), we have the equality
$$\frac{e^{r-2}(C)}{(d+1)!(r-2)!}=-\frac{(r-1)e(R/[I_1+\dots +I_r])}{d!(r-2)!}+\frac{E_{r-1}(I_1, \dots , I_r)}{(d+1)!(r-2)!}. $$
By multiplying $(d+1)!(r-2)!$ to the above equation, we get the desired formula. 
\end{proof}

As stated in the proof, the proof of Theorem \ref{main} contains our previous result in \cite{Ha2}. 
Moreover, the obtained formula for $e^{r-2}(C)$ can be viewed as a natural generalization 
of the Kirby-Rees formula given in \cite{KR2}. 

\begin{Corollary}\label{cor}
Let $I_1, \dots , I_r$ be $\fkm$-primary ideals in $R$ and assume that 
$I_1, \dots , I_{r-1} \subset I_r$, that is, the ideal $I_r$ is the largest 
ideal. Then we have a formula
$$e^{r-2}(R/I_1 \oplus \dots \oplus R/I_r)=e(R/[I_1+\dots +I_{r-1}] \oplus R/I_r). $$
In particular, if $I_1 \subset I_2 \subset \dots \subset I_r$, then 
$$e^{r-2}(R/I_1 \oplus \dots  \oplus R/I_r)=e(R/I_{r-1} \oplus R/I_r). $$
\end{Corollary}

\begin{proof}
Suppose that $I_1, \dots , I_{r-1} \subset I_r$. 
Then by Theorem \ref{main}, 
\begin{eqnarray*}
e^{r-2}(C) &=& \sum_{j=1}^r e( R/[I_1+\dots +\hat{I_j}+\dots +I_r]  
\oplus R/[I_1+\dots +I_r] ) \\
&& \hspace{5.5cm} -(d+1)(r-1)e(R/[I_1+\dots +I_r]) \\
&=& e(R/[I_1+ \dots +I_{r-1}] \oplus R/I_r)+(r-1)e(R/I_r \oplus R/I_r) \\
&& \hspace{7.5cm} -(d+1)(r-1)e(R/I_r) \\
&=& e(R/[I_1+\dots +I_{r-1}] \oplus R/I_r)+(r-1)(d+1)e(R/I_r) \\
&& \hspace{7.5cm} -(d+1)(r-1)e(R/I_r) \\
&=& e(R/[I_1+\dots +I_{r-1}] \oplus R/I_r). 
\end{eqnarray*}
Here the third equality follows from the elementary formula (\ref{ordinaryBuchsbaum-Rim}). 
\end{proof}

Before closing this article, we would like to give a few observations on the remaining multiplicities. 
We first recall the polynomial function $F_k(p, q)$ defined at the 
beginning of this section: 
$$F_k(p, q):=\Lambda(p, q)-\sum_{i=1}^k {q-(r-i)p-1 \choose r-i} g_{r-i+1}(p). $$
The key of our proof of Theorem \ref{main} is the fact that $\deg_q F_2(p, q) \leq r-3$ (Proposition \ref{limit}). It would be interesting to know whether this kind of property 
holds true or not for which $k$. 

\begin{Question}\label{question}
Let $p$ be a fixed large enough integer. Then for which $k=1, 2, \dots , r-1$, 
does the following hold true?
$$\lim_{q \to \infty} \frac{1}{q^{r-k}} F_k(p, q)=0. $$ 
In other word, is the degree of $F_k(p, q)$ with respect to $q$ at most $r-k-1$?
\end{Question}

This holds true when $k=2$ (and also $k=1$) by Proposition \ref{limit}. 
We are interested in the remaining cases. 
Suppose that $k \geq 3$. The affirmative answer to Question \ref{question} will tell us that 
for any $1 \leq j \leq k$, 
the $(r-j)$th associated Buchsbaum-Rim multiplicity $e^{r-j}(C)$ 
is determined by the polynomial
\begin{equation}
\sum_{i=1}^{k} {q-(r-i)p-1 \choose r-i} g_{r-i+1}(p). \label{expectedpoly} 
\end{equation}
Then we will be able to describe the multiplicity $e^{r-j}(C)$ 
as a sum of the ordinary Buchsbaum-Rim 
multiplicities of a direct sum of at most $(r-j)$ cyclic modules in the same manner.
Here we would like to record the expected formula. 
Note that the polynomial $g_{r-i+1}(p)$ defined in (\ref{polyg}) is of the form
$$g_{r-i+1}(p)=\frac{1}{(d+i-1)!} \sum_{\stackrel{A \subset [r]}{{}^{\sharp}A=i-1}} e(L_A) \cdot p^{d+i-1}+\mbox{(lower terms)}$$
where 
$
\displaystyle{
L_A:=\bigg( 
R\Big/
\Big[\sum_{s \in [r] \setminus A} I_s \Big]
\bigg)
\oplus 
\bigoplus_{j \in A}
\bigg(
R \Big/ \Big[ \sum_{s \in [r] \setminus A} I_s+I_j \Big]
\bigg)
}
$. 
We put
$$E_{r-i+1}:=E_{r-i+1}(I_1, \dots , I_r):=\sum_{\stackrel{A \subset [r]}{{}^{\sharp}A=i-1}} e(L_A).$$
Then for any $1 \leq j \leq k$, the coefficient of $p^{d+j-1}q^{r-j}$ in the polynomial (\ref{expectedpoly}) is 
$$\sum_{i=1}^j \frac{E_{r-i+1}}{(d+i-1)!(r-i)!} {r-i \choose r-j} \big( -(r-i) \big)^{j-i}. $$
If Question \ref{question} is affirmative, 
then the above coefficient coincides with 
$$\frac{e^{r-j}(C)}{(d+j-1)!(r-j)!} $$
so that we can get the formula for $e^{r-j}(C)$.   
Therefore, we can ask the following. 

\begin{Question}\label{conj}
Under the same notation as above, does the formula  
$$e^{r-j}(R/I_1 \oplus \dots \oplus R/I_r)=\sum_{i=1}^j {d+j-1 \choose j-i} \big( -(r-i) \big)^{j-i} E_{r-i+1}(I_1, \dots , I_r)$$
hold true?
\end{Question}

This is affirmative when $j=1$ (\cite[Theorem 1.3]{Ha}) and $j=2$ (Theorem \ref{main}). 
Note that the affirmative answer to Question \ref{question} for some $k$ implies the affirmative one to Question \ref{conj} for any $1 \leq j \leq k$. 
 
%%%%%%%%%%%%%%%%%%%%%%%%%%%%%%%%
%%%%%%%%%%%  References  %%%%%%%%%%%%%
%%%%%%%%%%%%%%%%%%%%%%%%%%%%%%%%

%%%%%%%%%%%%%%%%%%%%%%%%%%%%%%%%%%%%%%%%%
%%%%%%%%%%%%%%%%%%%%%%%%%%%%%%%%%%%%%%%%%
%%%%%%%%%%%%%%%%%%%%%%%%%%%%%%%%%%%%%%%%%


\begin{thebibliography}{99}

\bibitem{Bi} 
C. Bivi\`a-Ausina, 
The integral closure of modules, Buchsbaum-Rim multiplicities and Newton polyhedra, 
J. London Math. Soc. (2) 69 (2004), no. 2, 407--427


\bibitem{BR2}
D. A. Buchsbaum and D. S. Rim, 
A generalized Koszul complex. II. Depth and multiplicity, 
Trans. Amer. Math. Soc. 111 (1964), 197--224

\bibitem{CLU}
C.-Y. Jean Chan, J.-C. Liu, B. Ulrich, 
Buchsbaum-Rim multiplicities as Hilbert-Samuel multiplicities,  
J. Algebra 319 (2008), no. 11, 4413--4425


\bibitem{Ha}
F. Hayasaka, 
A computation of Buchsbaum-Rim functions of two variables in a special case, 
Rocky Mountain J. Math. 46 (2016), 1547--1557

\bibitem{Ha2}
F. Hayasaka, 
A formula for the associated Buchsbaum-Rim multiplicity of a direct sum of cyclic modules, to appear in J. Pure Appl. Algebra


\bibitem{J}
E. Jones, 
Computations of Buchsbaum-Rim multiplicities, 
J. Pure Appl. Algebra 162 (2001), no. 1, 37--52

\bibitem{KR1} D. Kirby and D. Rees, 
Multiplicities in graded rings. I. The general theory. 
Commutative algebra: syzygies, multiplicities, and birational algebra (South Hadley, MA, 1992), 209--267, 
Contemp. Math., 159, Amer. Math. Soc., Providence, RI, 1994

\bibitem{KR2} D. Kirby and D. Rees, 
Multiplicities in graded rings. II. Integral equivalence and the Buchsbaum-Rim multiplicity. 
Math. Proc. Cambridge Philos. Soc. 119 (1996), no. 3, 425--445

\bibitem{KT1}
S. Kleiman and A. Thorup, 
A geometric theory of the Buchsbaum-Rim multiplicity, 
J. Algebra 167 (1994), no. 1, 168--231

\bibitem{KT2}
S. Kleiman and A. Thorup, 
Mixed Buchsbaum-Rim multiplicities,  
Amer. J. Math. 118 (1996), no. 3, 529--569

\end{thebibliography}
\end{document}